\documentclass[12pt]{article}

\usepackage{amsthm}
\usepackage{amssymb}
\usepackage{amsmath}
\usepackage{color}
\usepackage[ruled, linesnumbered, noend]{algorithm2e}
\usepackage[hidelinks]{hyperref}

\newcommand{\pcl}{\Omega}

\newcommand{\PC}{\Pi}

\newcommand{\CC}{\mathcal{C}}

\theoremstyle{definition}
\newtheorem{theorem}{Theorem}

\newtheorem{lemma}{Lemma}
\newtheorem{proposition}{Proposition}

\title{Tight Bounds for Potential Maximal Cliques Parameterized by Vertex Cover}
\author{Tuukka Korhonen \\ {\small \texttt{tuukka.m.korhonen@helsinki.fi}} \\ {\small \url{https://tuukkakorhonen.com}}}

\begin{document}
\date{}
\maketitle

\begin{abstract}
We show that a graph with $n$ vertices and vertex cover of size $k$ has at most $4^k + n$ potential maximal cliques.
We also show that for each positive integer $k$, there exists a graph with vertex cover of size $k$, $O(k^2)$ vertices, and $\Omega(4^k)$ potential maximal cliques.
Our results extend the results of Fomin, Liedloff, Montealegre, and Todinca [Algorithmica, 80(4):1146--1169, 2018], who proved an upper bound of $poly(n) 4^k$, but left the lower bound as an open problem.
\end{abstract}

\section{Introduction}
Potential maximal cliques (PMCs) have been used to obtain large families of exact exponential time and FPT algorithms for graph problems~\cite{DBLP:journals/algorithmica/FominLMT18,DBLP:journals/siamcomp/FominTV15,korhonen_ipec,DBLP:journals/algorithmica/LiedloffMT19}.
These algorithms enumerate all PMCs of the input graph, so their time complexity is lower bounded by the number of PMCs.
In~\cite{DBLP:journals/algorithmica/FominLMT18} it was shown that graphs have $O^*(4^k)$\footnote{The $O^*(\cdot)$ notation suppresses factors polynomial in the number of vertices.} PMCs, where $k$ is the size of minimum vertex cover.
Algorithms with time complexity $O^*(4^k)$ were obtained via a corresponding PMC enumeration algorithm.
In the conclusion of~\cite{DBLP:journals/algorithmica/FominLMT18}, the authors remark that to them, the bound $O^*(4^k)$ for PMCs does ``not seem to be tight''.
In this paper we give an upper bound of $4^k + n$ for the number of PMCs.
We show that our bound is tight up to constant factors by giving a family of graphs with $\Omega(4^k)$ PMCs and $O(k^2)$ vertices.
Note that also the term $n$ is necessary, as $n$-vertex star graphs have $k=1$ and $n-1$ PMCs.

\section{Preliminaries}
The notation $S(n, k)$ denotes the Stirling numbers of the second kind, i.e., the number of ways to partition an $n$-element set into $k$ non-empty parts.

We consider graphs that are undirected and simple.
We denote the vertices of a graph $G$ with $V(G)$ and edges with $E(G)$.
The neighbors of a vertex $v$ are $N(v)$ and the neighborhood of a vertex set $X$ is $N(X) = \cup_{v \in X} N(v)$.
The closed neighborhood is denoted by $N[v] = N(v) \cup \{v\}$.
The set of vertex sets of connected components of $G$ is denoted by $\CC(G)$.
For a vertex subset $X \subseteq V(G)$, $G[X]$ denotes the subgraph of $G$ induced by $X$ and $G \setminus X$ the subgraph of $G$ induced by $V(G) \setminus X$.
A full component of a vertex set $X \subseteq V(G)$ is a component $C \in \CC(G \setminus X)$ such that $N(C) = X$.

A vertex subset is a PMC if it is a maximal clique in some minimal triangulation~\cite{DBLP:journals/siamcomp/BouchitteT01}.
Our proofs are based on the following characterization of PMCs.

\begin{proposition}[\cite{DBLP:journals/siamcomp/BouchitteT01}]
\label{pro:pmc}
Let $G$ be a graph.
A set $\pcl \subseteq V(G)$ is a PMC of $G$ if and only if
\begin{enumerate}
\item for each $C \in \CC(G \setminus \pcl)$, $N(C) \subsetneq \pcl$, i.e., $\pcl$ has no full components, and
\item for each pair of distinct vertices $u,v \in \pcl$, either $(u, v) \in E(G)$ or there is a component $C \in \CC(G \setminus \pcl)$ with $\{u,v\} \subseteq N(C)$.
\end{enumerate}
\end{proposition}

We refer to the first condition of Proposition~\ref{pro:pmc} as the \emph{no full component condition} and to the second condition as the \emph{cliquish condition}.
We call a PMC $\pcl$ of $G$ free if each of its vertices is adjacent to a vertex in $V(G) \setminus \pcl$.

\begin{lemma}
\label{lem:nonfree_ct}
If a PMC $\pcl$ is not free, then $\pcl = N[v]$ for some vertex $v$.
\end{lemma}
\begin{proof}
Let $v$ be a vertex of $\pcl$ that is not adjacent to any vertex in $V(G) \setminus \pcl$.
We have $N[v] \subseteq \pcl$ by definition and $N[v] = \pcl$ follows from the cliquish condition.
\end{proof}

In particular, the number of non-free PMCs is at most $n$, the number of vertices.

\section{Upper Bound}
Let $G$ be a graph and $V_k \subseteq V(G)$ a vertex cover of $G$ of size $k$.
We first reduce the task of upper bounding the number of PMCs of $G$ to this task in an induced supergraph $M(G, V_k)$ of $G$.
Then we show that the PMCs of $M(G, V_k)$ have a relatively simple structure that we can handle with some case analysis.

The following lemma allows us to reduce to an induced supergraph.

\begin{lemma}[\cite{DBLP:journals/tcs/BouchitteT02}]
\label{lem:ind_sub}
Let $G$ be a graph and $X \subseteq V(G)$.
The number of PMCs of $G[X]$ is at most the number of PMCs of $G$.
\end{lemma}

Now we add $2^k-1$ vertices to $G$ to get $M(G, V_k)$.
In particular, for each non-empty subset $X \subseteq V_k$ we add a vertex $M_X$ with $N(M_X) = X$.
The vertices $V_k$ form a vertex cover of $M(G, V_k)$, and we refer to them as the inner vertices of $M(G, V_k)$.
We refer to other vertices of $M(G, V_k)$ as the outer vertices.

Now we start to expose the structure of PMCs of $M(G, V_k)$.
\begin{lemma}
Let $\pcl$ be a PMC of $M(G, V_k)$.
The vertices $V_k$ intersect at most three distinct connected components of $M(G, V_k) \setminus \pcl$.
\end{lemma}
\begin{proof}
Suppose that there are four disjoint subsets $P_1, P_2, P_3, P_4$ of $V_k$, each a subset of a distinct connected component of $M(G, V_k) \setminus \pcl$.
The vertices $M_{P_1 \cup P_2}$ and $M_{P_3 \cup P_4}$ are both in $\pcl$, because otherwise $P_1$ would be connected to $P_2$ and $P_3$ to $P_4$.
However, the pair of vertices $M_{P_1 \cup P_2}$, $M_{P_3 \cup P_4}$ violates the cliquish condition because the neighborhoods of them are in different connected components of $M(G, V_k) \setminus \pcl$, which is a contradiction.
\end{proof}

Now, given a PMC $\pcl$, the inner vertices $V_k$ can be partitioned into at most four parts: the intersection $V_k \cap \pcl$, and at most three connected subsets.
We denote this partition by $P(\pcl) = \{V_k \cap \pcl, P_1, P_2, P_3\}$, or by subsets of it if some parts are missing.
Our goal is to show that given this partition, we can (almost) determine for each vertex of $M(G, V_k)$ whether it is in $\pcl$ or not.

First, fixing $V_k \cap \pcl$ determines which of the inner vertices are in $\pcl$.
Second, if the neighborhood of an outer vertex intersects multiple of the components $P_i$, then it must be in the PMC because otherwise these would not be disjoint components.
Third, we have the following lemma.

\begin{lemma}
\label{lem:subsv}
Let $\pcl$ be a PMC of $M(G, V_k)$ and $P_i \subseteq V_k \setminus \pcl$ a part of $P(\pcl)$.
If $v$ is an outer vertex of $M(G, V_k)$ with $N(v) \subseteq P_i$, then $v \notin \pcl$.
\end{lemma}
\begin{proof}
Suppose that $v \in \pcl$.
All neighbors of $v$ are in a same component of $M(G, V_k) \setminus \pcl$.
Therefore, if $\pcl$ would satisfy the cliquish condition, then $P_i$ would be a full component, which is a contradiction.
\end{proof}

By Lemma~\ref{lem:subsv}, $P(\pcl)$ determines the outer vertices $v$ with $N(v) \subseteq P_i$ for $i \in \{1,2,3\}$.
Now, to determine outer vertices $v$ with $N(v) \subseteq V_k \cap \pcl$ we assume that the PMC $\pcl$ is free.
It follows that these vertices must not be in the PMC because otherwise the PMC would subsume their neighborhood, and therefore be non-free.

What is left undetermined are outer vertices whose neighborhood intersects $V_k \cap \pcl$ and a single part $P_i$.
To finish the argument, we consider different cases based on what parts there are in $P(\pcl)$.

If $P(\pcl) = \{P_1, P_2, P_3\}$ we are already ready, and the number of such PMCs is $S(k, 3)$.
There are no PMCs with $P(\pcl) = \{P_1, P_2\}$ or $P(\pcl) = \{P_1\}$ because they would have a full component.
Also, there are no free PMCs $\pcl$ with $P(\pcl) = \{V_k \cap \pcl\}$ because the vertex $M_{V_k}$ would be forced in $\pcl$ making it non-free.
The remaining cases are that we have a non-empty part $V_k \cap \pcl$, and at least one but at most three other parts $P_i$.

\begin{lemma}
\label{lem:p1}
For a fixed partition $\{V_k \cap \pcl, P_1\}$ of $V_k$, there are at most $k$ free PMCs $\pcl$ of $M(G, V_k)$ such that $P(\pcl) = \{V_k \cap \pcl, P_1\}$.
\end{lemma}
\begin{proof}
Let $C_1$ be the component of $M(G, V_k) \setminus \pcl$ with $P_1 \subseteq C_1$.
To satisfy the no full component condition, there must at least one vertex $v \in V_k \cap \pcl$ that is not in $N(C_1)$.
Next we prove that fixing such a vertex $v$ uniquely determines $\pcl$.
Each outer vertex $u$ whose neighborhood intersects both $P_1$ and $v$ must be in $\pcl$.
Each outer vertex $u$ whose neighborhood intersects $P_1$ and $V_k \cap \pcl$ but not $v$ must not be in $\pcl$ because otherwise it would violate the cliquish condition with $v$.
\end{proof}

By Lemma~\ref{lem:p1}, there are at most $k 2^k$ free PMCs $\pcl$ that induce a partition of type $P(\pcl) = \{V_k \cap \pcl, P_1\}$.

\begin{lemma}
\label{lem:p2}
For a fixed partition $\{V_k \cap \pcl, P_1, P_2\}$ of $V_k$, there are at most $2k$ free PMCs $\pcl$ of $M(G, V_k)$ such that $P(\pcl) = \{V_k \cap \pcl, P_1, P_2\}$.
\end{lemma}
\begin{proof}
Recall that only the outer vertices whose neighborhoods intersect $V_k \cap \pcl$ and exactly one of $P_1$ and $P_2$ are undetermined.
Consider two such vertices, one whose neighborhood intersects $P_1$ and one whose neighborhood intersects $P_2$.
They both cannot be in $\pcl$ because they would violate the cliquish condition.
Now we have two options, to add only neighbors of $P_1$ in $\pcl$ or only neighbors of $P_2$.
Let's assume that we choose $P_1$ and multiply the count by two.
Now we are left with the exact same situation as in Lemma~\ref{lem:p1}, and by the same arguments we get at most $k$ PMCs.
\end{proof}

By Lemma~\ref{lem:p2}, there are at most $6k \cdot S(k, 3)$ free PMCs $\pcl$ that induce a partition of type $P(\pcl) = \{V_k \cap \pcl, P_1, P_2\}$.

\begin{lemma}
\label{lem:p3}
For a fixed partition $\{V_k \cap \pcl, P_1, P_2, P_3\}$ of $V_k$, there is at most one free PMC $\pcl$ of $M(G, V_k)$ such that $P(\pcl) = \{V_k \cap \pcl, P_1, P_2, P_3\}$.
\end{lemma}
\begin{proof}
Let $v$ be an outer vertex whose neighborhood intersects $V_k \cap \pcl$ and $P_i$.
The vertex $v$ cannot be in $\pcl$, because it would violate the cliquish condition with $M_{P_j \cup P_l}$ where $i,j,l$ are distinct.
\end{proof}

By Lemma~\ref{lem:p3}, there are at most $4 \cdot S(k, 4)$ free PMCs $\pcl$ that induce a partition of type $P(\pcl) = \{V_k \cap \pcl, P_1, P_2, P_3\}$.

\begin{theorem}
A graph with vertex cover of size $k$ has at most $4^k + n$ PMCs.
\end{theorem}
\begin{proof}
Let $G$ be a graph with vertex cover $V_k$ of size $k$.
By Lemma~\ref{lem:ind_sub} it suffices to bound the number of PMCs of $M(G, V_k)$.
The number of vertices of $M(G, V_k)$ is $n + 2^k-1$, so the number of non-free PMCs of it is at most $n + 2^k - 1$ by Lemma~\ref{lem:nonfree_ct}.
By Lemmas~\ref{lem:p1}, \ref{lem:p2}, and \ref{lem:p3}, the number of free PMCs of $M(G, V_k)$ is at most $S(k, 3) + k 2^k + 6k \cdot S(k, 3) + 4 \cdot S(k, 4)$.
By computing the values for small $k$ and observing that $4 \cdot S(k, 4) \le 4^k / 6$ we can verify that this sums up to at most $4^k + n$ in total.
\end{proof}

\section{Lower Bound}
We consider graph $G_k$, whose vertex set is $[k] \cup \binom{[k]}{2}$ consisting of the set $[k]$ of positive integers up to $k$ and the set $\binom{[k]}{2}$ of unordered pairs of positive integers up to $k$.
The edge set of $G_k$ is defined by connecting each vertex $(i,j) \in \binom{[k]}{2}$ to vertices $i$ and $j$ in $[k]$.
Therefore both $[k]$ and $\binom{[k]}{2}$ are independent sets and $[k]$ is a vertex cover of $G_k$.

We first show that $G_k$ has at least $S(k, 3) = \Omega(3^k)$ free PMCs $\pcl$ with $\pcl \cap [k] = \emptyset$.
Then we extend this bound by induction to show that for each $i \le k$, $G_k$ has at least $S(k-i, 3)$ free PMCs $\pcl$ with $\pcl \cap [k] = [i]$.
By symmetry and binomial theorem, this implies the lower bound $\Omega(4^k)$.

\begin{lemma}
\label{lem:base}
The graph $G_k$ has at least $S(k, 3)$ free PMCs $\pcl$ with $\pcl \cap [k] = \emptyset$.
\end{lemma}
\begin{proof}
Consider a tripartition $P = \{P_1,P_2,P_3\}$ of $[k]$.
Construct a PMC $\pcl$ such that $(i, j) \in \pcl$ if and only if $i$ and $j$ are in different parts of $P$.
Note that from such $\pcl$ we can uniquely recover $P$ because vertices in $\binom{[k]}{2} \setminus \pcl$ are grouped into three components.
Now it suffices to prove that $\pcl$ is indeed a PMC, and it follows that the number of such PMCs is at least the number of tripartitions, i.e., $S(k, 3)$.

The graph $G_k \setminus \pcl$ has exactly tree connected components, each consisting of vertices $P_i \cup \binom{P_i}{2}$ for $i \in \{1,2,3\}$.
None of the components are full, because for the component $P_i \cup \binom{P_i}{2}$ each vertex of form $v \in P_j \times P_l$ (with $i,j,l$ distinct) is in $\pcl$ but not in the neighborhood of $P_i \cup \binom{P_i}{2}$.
Each vertex $v \in \pcl$ has exactly two adjacent components because it is of form $v \in P_i \times P_j$ (with $i \neq j$).
Therefore $\pcl$ is free, and any pair of vertices in $\pcl$ has a common adjacent component so $\pcl$ satisfies the cliquish condition.
\end{proof}

Let $\PC_f(G_k, i)$ denote the set of free PMCs $\pcl$ of $G_k$ with $\pcl \cap [k] = [i]$.
By Lemma~\ref{lem:base} we have $|\PC_f(G_k, 0)| \ge S(k, 3)$.

\begin{lemma}
\label{lem:ind}
It holds that $|\PC_f(G_k, i)| \ge |\PC_f(G_{k-1}, i-1)|$.
\end{lemma}
\begin{proof}
Let $\pcl'$ be a free PMC of $G_{k-1}$ with $\pcl' \cap [k-1] = [i-1]$.
We claim that $\pcl = \pcl' \cup \{k\}$ is a free PMC of $G_k$.
The inequality follows from this claim by permuting the indices appropriately.

First, note that $G_k \setminus \pcl$ can be obtained from $G_{k-1} \setminus \pcl'$ by adding the vertices $(j, k)$ for $j<k$.
These vertices have degree one in $G_k \setminus \pcl$, so their addition can not create a full component, so $\pcl$ does not have a full component in $G_k$ because $\pcl'$ does not have a full component in $G_{k-1}$.
For the cliquish condition, the vertex $k$ is connected to each vertex in $[i-1]$ via the new vertices $(j,k)$.
Other vertices in $\pcl'$ are of form $(a,b)$, but either $a \notin \pcl'$ or $b \notin \pcl'$ because $\pcl'$ is free, so $k$ can reach them via either $(a, k)$ or $(b,k)$.
Finally, $\pcl$ is free because $(1, k) \notin \pcl$.
\end{proof}

\begin{theorem}
The number of PMCs of $G_k$ is $\Omega(4^k)$.
\end{theorem}
\begin{proof}
By combining Lemmas~\ref{lem:base} and~\ref{lem:ind}, we get that $|\PC_f(G_k, i)| \ge S(k-i, 3)$.
By symmetry, the total number of PMCs of $G_k$ is at least $$\sum_{i=0}^k \binom{k}{i} |\PC_f(G_k, i)| \ge \sum_{i=0}^k \binom{k}{i} S(k-i, 3) \ge 3^{-3} \sum_{i=0}^{k-3} \binom{k}{i} 3^{k-i} = \Omega(4^k).$$
\end{proof}

\section{Conclusion}
We showed that a graph with $n$ vertices and vertex cover of size $k$ has at most $4^k + n$ PMCs.
Furthermore, this bound is tight up to constant factors.
Our result refines the upper bound $O^*(4^k)$ given in~\cite{DBLP:journals/algorithmica/FominLMT18} and closes the open problem of giving a lower bound.

We hope that in addition to showing limitations of the PMC approach, our proofs could yield positive insights on PMCs.
Our proofs show that in the worst case, the vertex cover is an independent set, and that intuitively, the more edges there are within the vertex cover, the less PMCs there are.
This seems to generalize the graph class of split graphs, which have a clique vertex cover and at most $n$ PMCs.

\bibliographystyle{plain}
\bibliography{paper}

\end{document}